\documentclass[]{article}
\usepackage[english]{babel}
\usepackage[normalem]{ulem}
\usepackage[all]{xy}
\usepackage{amsmath,amsthm,booktabs,color,dsfont,epsfig,graphicx,mathrsfs,multirow,scalefnt}
\usepackage{amsfonts, amssymb}

\newtheorem{theo}{Theorem}[section]
\newtheorem{cor}[theo]{Corollary}
\newtheorem{lem}[theo]{Lemma}

\newtheorem{prop}[theo]{Proposition}
\newtheorem{defn}[theo]{Definition}
\newtheorem{rmk}[theo]{Remark}
\newtheorem{ex}[theo]{Example}
\newcommand{\Z}{\mathbb{Z}}
\newcommand{\N}{\mathbb{N}}

\newcommand{\s}{\sigma}
\newcommand{\LA}{L_\mathbf{\Lambda}}

\newcommand{\F}{\mathcal{F}}

\newcommand{\Pp}{\mathcal{P}}

\title{Sliding block codes between shift spaces over infinite alphabets}

\author{
\small{Daniel Gon\c{c}alves}\\
\footnotesize{UFSC -- Department of Mathematics}\\
\footnotesize{88040-900 Florian\'{o}polis - SC, Brazil}\\
\footnotesize{\texttt{daemig@gmail.com}}
\and
\small{Marcelo Sobottka}\\
\footnotesize{UFSC -- Department of Mathematics}\\
\footnotesize{88040-900 Florian\'{o}polis - SC, Brazil}\\
\footnotesize{\texttt{sobottka@mtm.ufsc.br}}
\and
\small{Charles Starling}\\
\footnotesize{University of Ottawa}\\
\footnotesize{Dept. of Mathematics and Statistics}\\
\footnotesize{585 King Edward, Ottawa, ON}\\
\footnotesize{K1N 6N5}\\
\footnotesize{\texttt{cstar050@uottawa.ca}}
}

\date{}

\begin{document}

\maketitle

\begin{abstract} Recently Ott, Tomforde and Willis introduced a notion of one-sided shifts over infinite alphabets and proposed a definition for sliding block codes between such shift spaces.

In this work we propose a more general definition for sliding block codes between Ott-Tomforde-Willis shift spaces and then we prove Curtis-Hedlund-Lyndon type theorems for them, finding sufficient and necessary conditions under which the class of the sliding block codes coincides with the class of continuous shift-commuting maps.
\end{abstract}

\bigskip
\hrule
\noindent
{\footnotesize\em This is a pre-copy-editing, author-produced PDF of an article accepted for publication in Mathematische Nachrichten, following peer review. The definitive publisher-authenticated version {\em D. Gon\c{c}alves,  M. Sobottka and C. Starling. Sliding Block Codes between Shift Spaces over Infinite Alphabets. (2016), 289, 17, 2178-2191; Math. Nachr.. doi:10.1002/mana.201500309}, is available online at:\break http://onlinelibrary.wiley.com/doi/10.1002/mana.201500309/abstract .}
\hrule
\bigskip

% ---------------------------- INTRODUCTION ------------------------------------

\section{Introduction}

Symbolic dynamics is a fundamental area of dynamical systems, and is typically concerned with a finite alphabet $A$, spaces of infinite sequences $A^\N$ or $A^\Z$ in $A$, and closed subspaces of $A^\N$ or $A^\Z$ which are invariant under the shift map. Though most of the theory is developed for shift spaces over finite alphabets, over the last decades researchers have proposed generalizations for the infinite alphabet case, see for example  \cite{FiebigFiebig1995,kitchens1997,LindMarcus}. One of the hurdles to overcome when defining shift spaces over infinite alphabets is the fact that the countable product of an infinite discrete space is not locally compact.

One approach to dealing with this difficulty was taken in \cite{FiebigFiebig1995}, where the authors use the Alexandroff one-point compactification for locally compact shift spaces over countable alphabets. Such compactification was used to prove several results for the entropy of countable Markov shifts \cite{Fiebig2001,Fiebig2003,FiebigFiebig1995,FiebigFiebig2005}.

In a recent paper by Ott, Tomforde, and Willis, see \cite{Ott_et_Al2014}, the authors constructed a compactification $\Sigma_A$ of $A^\N$ which can be identified with the set of all finite sequences in $A$ (including an empty sequence) together with the infinite sequences in $A$. Subshifts in this case are then taken to be closed subsets $\Lambda\subset \Sigma_A$ invariant under the shift which satisfy the so called ``infinite extension property'', which guarantees that $\Lambda$ contains infinite sequences and is determined by its infinite sequences. This is an approach taken with an eye towards applications in the C*-algebras of infinite graphs and \cite{Ott_et_Al2014} represents a comprehensive study of such shifts, morphisms between them, and related groupoids.

While the one-point compactification can only be used for locally compact shift spaces (which constitute the class of shift spaces that are simultaneously row-finite and column-finite -- see the end of Section \ref{Background}), the Ott-Tomforde-Willis compactification can be used for any shift space and coincides with the Alexandroff compactification if the shift is simultaneously row-finite and column-finite. However, a cost of using the Ott-Tomforde-Willis compactification is that it generally introduces infinitely many new points in the space.

In the classical situation, a key fact is that any continuous map $\Phi:\Lambda\to \Gamma$ between shift spaces which commutes with the shift map must be a {\em sliding block code}, that is, there is a positive integer $n$ such that the $i$th entry of the sequence $\varphi(x)$ depends only on the entries of $x$ in a window of size $n$ around $x_i$ -- this is the Curtis-Hedlund-Lyndon theorem.

In this work we propose a definition of sliding block codes between Ott-Tomforde-Willis shift spaces over countable alphabets, and characterize them. We note that our definition is more general than the one given in \cite{Ott_et_Al2014} -- see Definition \ref{defn_sliding block code}). We show that when $A$ is finite our definition matches the classical definition, and  also present sufficient and necessary conditions under which the class of continuous shift-commuting maps coincide with the class of sliding block codes, obtaining an analogue to the Curtis-Hedlund-Lyndon theorem.

We remark that given a two-sided shift space over infinitely countable alphabets, one can construct a compactification of it in an  analogous way to that used in \cite{Ott_et_Al2014} for one-sided shifts. However, the resulting topology is not metrizable and the problem of obtaining general Curtis-Hedlund-Lyndon type results for them remains open (see \cite{GSS1}).

What follows is broken into two sections. In Section \ref{Background} we recall background and definitions from \cite{Ott_et_Al2014}. In Section \ref{sliding block code} we define sliding block codes and undertake a comprehensive study of them, culminating in the previously mentioned versions of the Curtis-Hedlund-Lyndon theorem for such sliding block codes (Subsection \ref{CHL-Theo}). To close, we discuss the notions of higher order shift presentations.

\section{Background}\label{Background}

We will denote the set of all positive integers by $\N$.

Recall that a locally compact Hausdorff space is called  {\em zero dimensional} or {\em totally disconnected} if it has a basis consisting of clopen sets. A {\em dynamical system} is a pair $(X,T)$, where $X$ is a locally compact space and $T:X\to X$ is a map. If $T$ is continuous we say that $(X,T)$ is a {\em topological dynamical system}. Two dynamical systems $(X,T)$ and $(Y,S)$ are {\em conjugate} if there exists a bijective map $\Phi:X\to Y$ such that $\Phi\circ T=S\circ \Phi$. If, in addition, $(X,T)$ and $(Y,S)$ are topological dynamical systems and $\Phi$ is a homeomorphism then $(X,T)$ and $(Y,S)$ are {\em topologically conjugate}. In the above case, we call the $\Phi$ a {\em (topological) conjugacy} for the dynamical systems. When $X$ and $Y$ are shift spaces and $T$ and $S$ are their respective shift maps we usually just say that $X$ and $Y$ are (topologically) conjugate to refer that the dynamical system are (topologically) conjugate.

%\subsection{Dynamical systems}

%\begin{defn} A locally compact Hausdorff space is said to be {\em zero dimensional} or {\em totally disconnected} if it has a base consisting of clopen sets.
%\end{defn}

%\begin{defn} A {\em dynamical system (DS)} is a pair $(X,T)$, where $X$ is a locally compact space and $T:X\to X$ is a map. If $T$ is continuous we will say that $(X,T)$ is a {\em topological dynamical system
%(TDS)}.\end{defn}

%\begin{defn}\label{TDS_conjugate}
%We say that two DS $(X,T)$ and $(Y,S)$ are {\em conjugate} if there exists a bijective map $\Phi:X\to Y$ such that $\Phi\circ T=S\circ \Phi$. If in addition $(X,T)$ and $(Y,S)$ are TDS and $\Phi$ is a homeomorphism, we say
%that $(X,T)$ and $(Y,S)$ are {\em topologically conjugate}. The map $\Phi$ is called a {\em (topological) conjugacy} for the dynamical systems.
%\end{defn}

%\begin{defn} We say that a TDS $(X,T)$ is {\em expansive} if there exists a partition $\mathcal{U}$ of $X$ such that, for all $x\neq y$ in $X$, there exists $n \geq 0$ such that $T^n(x)$ and $T^n(y)$ are in different
%elements of $\mathcal{U}$. In such a case, $\mathcal{U}$ is called an {\em expansive partition for $T$}.
%\end{defn}

\subsection{Ott-Tomforde-Willis one-sided shift spaces}

Fix an infinite countable discrete set $A$ and define a new symbol $\o$, which we call the {\em empty letter}, and define the extended alphabet $\tilde{A}:=A\cup\{\o\}$. We will consider the set
of all infinite sequences over $\tilde{A}$, $\tilde{A}^\N:=\{(x_i)_{i\in\N}: x_i\in \tilde{A}\}$
and define $\Sigma_A^{\text{inf}},\Sigma_A^{\text{fin}}$ by
$$\Sigma_A^{\text{inf}}:=A^\N:=\{(x_i)_{i\in\N}:\ x_i\in A\ \forall i\in \N\}$$ and $$\Sigma_A^{\text{fin}}:=\{(x_i)_{i\in\N}:\ x_i \in \tilde{A} \text{ and if } x_i=\o\ \text{then} \ x_{i+1}=\o, \forall i\in\N\}.$$

\begin{defn}The {\em Ott-Tomforde-Willis one-sided full shift} over $A$ is the set $$\Sigma_A:=\left\{\begin{array}{lcl}\Sigma_A^{\text{inf}} & \text{, if} & |A|<\infty\\\\ \Sigma_A^{\text{inf}}\cup\Sigma_A^{\text{fin}} & \text{, if} &
|A|=\infty\end{array}\right..$$

We define the length of $x\in\Sigma_A$ as $l(x):=\min_{k\geq 1}\{k-1: x_k=\o\}$.
\end{defn}

Note that $l(x)<\infty$ if and only if $x\in \Sigma_A^{\text{fin}}$ (in particular, $\O$ is the unique sequence with length zero). We will refer the  constant sequence $\O:=(\ldots\o\o\o \ldots)\in \Sigma_A^{\text{fin}}$ as the {\em empty sequence}, and, in general, the sequences of the set $\Sigma_A^{\text{fin}}$ will be referred as {\em finite sequences}. To simplify the notation we will identify $\Sigma_A^{\text{fin}}$ with the set $\{\O\}\cup\bigcup_{k\in\N}A^k$ by identifying $(x_1x_2x_3\ldots x_k\o\o\o\o\ldots)$ with $(x_1x_2x_3\ldots x_k)$. Following this identification, we will use the notation $x=(x_i)_{i\leq k}$ to refer to a point of $\Sigma_A^{\text{fin}}$ with length $k\geq 1$.

Given two elements $x = (x_1, \dots, x_m)$ and $ y = (y_1, \dots y_n)\in \Sigma_A^{\text{fin}}$, we may form their {\em concatenation} $xy = (x_1, \dots, x_m, y_1, \dots y_n)$. We may similarly concatenate an element $x\in \bigcup_{k\in\N}A^k$ with an element $y\in A^\N$ to form $xy\in A^\N$. We will also use the convention that $x\O=x$ for all $x\in\Sigma_A^{\text{fin}}$.
%************************************************************************
% Retirado do paragrafo acima: $x=\O x$ for all sequences $x\in\Sigma_A$
% Pois da problema com a definicao de conjuntos finitamente definidos
%*******************************************************************

Given $x\in\Sigma_A^{\text{fin}}$ and a finite set $F\subset A$, a {\em generalized cylinder set} of $\Sigma_A$ is defined as:
\begin{equation}\label{cylinder}Z(x,F):=\left\{\begin{array}{lcl}\{y\in\Sigma_A:\ y_i=x_i\ \forall i=1,\ldots, l(x),\ y_{l(x)+1}\notin F\} &\text{, if}& x\neq\O\\
                          \{y\in\Sigma_A:\ y_1\notin F\} &\text{, if}& x=\O.\end{array}\right.\end{equation}

If $F=\emptyset$ we will say that $Z(x,F)$ has length $k$, otherwise we will say that $Z(x,F)$ has length $k+1$. Furthermore, if $F=\emptyset$ then the set $Z(x,F)$ is simply denoted by $Z(x)$. Notice that the sets of the form $Z(x)$ coincide with the usual cylinders which form the basis of the product topology in $\Sigma_A^{\text{inf}}$, while $Z(\O,F)$ is a neighborhood of the empty sequence $\O$. Furthermore, the generalized cylinders are clopen and, for each finite set $F$, $Z(\O,F)$ corresponds to the complement of a finite union of cylinders of the form $Z(x)$.

The full shift space over $A$ from \cite{Ott_et_Al2014} is the set $\Sigma_A$ given the topology generated by the generalized cylinder sets. With this topology $\Sigma_A$ is zero dimensional, compact and metrizable (see
\cite{Ott_et_Al2014}, Section 2).

The {\em shift map} $\s:\Sigma_A\to\Sigma_A$ is defined as the map given by $\s((x_i)_{i\in\N})=(x_{i+1})_{i\in\N}$ for all $(x_i)_{i\in\N}\in\Sigma_A$. Note that
$\s(\O)=\O$ and for all $x=(x_1\ldots x_k)\in\Sigma_A^{\text{fin}}$, $x\neq \O$, we have that $\s(x_1x_2x_3\ldots x_k)=(x_2x_3x_4\ldots x_k)$.

We remark that, if $A$ is infinite, then $\s$ is not continuous at the empty sequence $\O$ (\cite{Ott_et_Al2014}, Proposition 2.23).

Given a subset $\Lambda\subseteq\Sigma_A$, let \begin{equation}\label{shift_inf_fin}\begin{array}{lcl}\Lambda^{\text{fin}}&:=&\Lambda\cap\Sigma_A^{\text{fin}}\\\\ \text{and}\\\\
\Lambda^{\text{inf}}&:=&\Lambda\cap\Sigma_A^{\text{inf}},\end{array}\end{equation}
the set of all finite sequences of $\Lambda$ and the set of all infinite sequences of $\Lambda$, respectively. Define  $B_\infty(\Lambda)=\Lambda^{\text{inf}}$ and, for each $n\geq 1$, let
\begin{equation}\label{subblocks}B_n(\Lambda):=\{(x_1,\ldots,x_n)\in {\tilde{A}}^n: x_1,\ldots,x_n\text{ is a subblock of some sequence in }\Lambda\}.\end{equation}
These are the {\em blocks of length $n$} in $\Lambda$. We single out the blocks of length one and use the notation $\LA:= B_1(\Lambda)\setminus\{\O\}$ -- this is the set of all symbols used by sequences of $\Lambda$, or the {\em letters} of $\Lambda$. The {\em language} of $\Lambda$ is \begin{equation}\label{language}B(\Lambda):=\bigcup_{n\geq 1}B_n(\Lambda).\end{equation}

\begin{rmk} Notice that $(x_1\ldots x_k)\in\Lambda^{\text{fin}}$ stands for a finite sequence $(x_i)_{i\in\N}\in\Lambda$ with length $k$ so that $x_1,\ldots,x_k\in A$ and $x_i=\o$ for all $i>k$, while a block $(y_1\ldots y_k)\in B(\Lambda)$ stands for a finite subblock of some sequence of $\Lambda$ and therefore $y_1,\ldots,y_k\in \tilde{A}$ (with the condition that if $y_i=\o$ then $y_j=\o$ for all $i\leq j\leq k$).

\end{rmk}

We now recall the notion of a shift space from \cite{Ott_et_Al2014}. We first define the {\em follower set} and the {\em predecessor set} of a block $a\in B_n(\Lambda)$ (for $1\leq n <\infty$) in some set $\Lambda\subset \Sigma_A$ as

\begin{equation}\label{followersets}\F(\Lambda,a):=\{b\in B_1(\Lambda):\ ab\in B_{n+1}(\Lambda)\}\end{equation}
and
\begin{equation}\label{predecessorsets}\Pp(\Lambda,a):=\{b\in B_1(\Lambda):\ ba\in B_{n+1}(\Lambda)\},\end{equation}
respectively.

A subset $\Lambda\subseteq\Sigma_A$ is called a {\em shift space} over $A$ if the following three properties hold:

\begin{enumerate}
\item[1 -] $\Lambda$ is closed with respect to the topology of $\Sigma_A$;
\item[2 -] $\Lambda$ is invariant under the shift map, that is, $\s(\Lambda)\subseteq\Lambda$;
\item[3 -] $\Lambda$ satisfies the ``infinite extension property'', that is, $\O$ belongs to $\Lambda^{\text{fin}}$ if, and only if, $|\LA|=\infty$, while a finite sequence $x\neq\O$ belongs to $\Lambda^{\text{fin}}$ if, and only if, $|\F(\Lambda,x)|=\infty$.
\end{enumerate}

\begin{rmk} The equivalence between the definition given above of ``infinite extension property'' and the definition given in \cite{Ott_et_Al2014} follows from Proposition 3.7 in \cite{Ott_et_Al2014}.
\end{rmk}

From the definition of the subspace topology we have that cylinder sets of the form $Z(x,F)\cap\Lambda$ generate the topology in $\Lambda$.
Notice that properties 1 and 2 above are exactly the same ones that define a subshift over a finite alphabet. Property 3 assures that there always exist infinite sequences in a non empty shift space. In fact, it is
possible to prove that $\Lambda^{\text{inf}}$ is dense in $\Lambda$ (see Proposition 3.8 in \cite{Ott_et_Al2014}). % We remark that $\Lambda$ is compact and hence, by taking the restriction of the shift map to $\Lambda$, $\s:\Lambda\to\Lambda$, the pair $(\Lambda,\s)$ is a  dynamical system.

An equivalent way to define a shift space $\Lambda\subset\Sigma_A$ is via a set of forbidden words. In fact, given $\mathbf{F}\subset \bigcup_{k\in \N}A^k$ we define \begin{equation}\label{forbidden_words}X_{\mathbf{F}}^{\text{inf}} =
\{ x \in \Sigma_A^{\text{inf}}: B(\{x\})\cap\mathbf{F}=\emptyset \} \text{, } X_{\mathbf{F}}^{\text{fin}} = \{ x \in B(X_{\mathbf{F}}^{\text{inf}}): |\F (X_{\mathbf{F}}^{\text{inf}},x)|=\infty \}
\end{equation} and $ X_{\mathbf{F}}:=X_{\mathbf{F}}^{\text{inf}} \cup X_{\mathbf{F}}^{\text{fin}}$.
In other words, $X_{\mathbf{F}}$ is the set of all infinite sequences that do not have a subblock in $\mathbf{F}$, union with the set of all finite sequences which satisfy the `infinite extension property'. Theorem 3.16 in \cite{Ott_et_Al2014} assures that $\Lambda\subset\Sigma_A$ is a shift space if and only if $\Lambda = X_{\mathbf{F}}$ for some $\mathbf{F}\subset \bigcup_{k\in \N}A^k$.

Finally, we remark that, whenever $L_\Lambda$ is finite, the definition of the Ott-Tomforde-Willis shift space $\Lambda$ coincides with the
standard definition of shift spaces over finite alphabets (\cite{Ott_et_Al2014},
Proposition 3.19).

We will say that a given nonempty shift $\Lambda=X_\mathbf{F}$ is a:\\

 {{\sc shift of finite type (SFT)}}: if $\mathbf{F}$ is finite;\\

 {{\sc $M$-step shift}}: if $\mathbf{F}\subseteq A^{M+1}$;\\

 {{\sc row-finite shift}}: if for all $a\in\LA$ we have that $\F(\Lambda,a)$ is a finite set (which is equivalent to $\Lambda^{\text{fin}}\subset\{\O\}$ \cite[Proposition 3.21]{Ott_et_Al2014});\\

 {{\sc column-finite shift}}: if for all $a\in\LA$ we have that $\Pp(\Lambda,a)$ is a finite set.\\

For standard shift spaces (over a finite alphabet) all the above classes coincide. However, this is not true when the alphabet is infinite (\cite{Ott_et_Al2014}, Remark 5.19). Furthermore, under the Ott-Tomforde-Willis definition of shift conjugacy, which requires a conjugacy to be length preserving, it has been shown in \cite{GR} that there are (M+1) step shifts that are not conjugate to M-step shifts.

\section{Sliding block codes for shift spaces over countable alphabets}\label{sliding block code}

Classically, sliding block codes are the most important type of maps between shift spaces, having a major role in coding theory. Sliding block codes between standard shift spaces over finite alphabets are maps defined from a local rule. More precisely, given two finite alphabets $A$ and $B$ and $\Lambda\subseteq A^\N$ and $\Gamma\subseteq B^\N$ shift spaces, a map $\Phi:\Lambda\to\Gamma$ is a sliding block code if there exist $\ell\geq 0$ and a map $\phi:B_{\ell+1}(\Lambda)\to L_\Gamma$ such that $\big(\Phi(x)\big)_n=\phi(x_{n}\ldots x_{n+\ell})$, for all $n\in\N$. In such a case, $\ell$ is called the anticipation of $\Phi$. We recall that sliding block codes between standard shift spaces over finite alphabets coincide with the class of continuous maps that commute with the shift map (this is the Curtis-Hedlund-Lyndon Theorem).

 The natural approach to generalizing the notion of sliding block codes for the infinite alphabet case is to say that $\Phi:\Lambda\to\Gamma$ is a sliding block code if there exists $\ell\geq 0$ and a local rule $\phi:B_{\ell+1}(\Lambda)\to L_\Gamma$ which acts as explained above \cite[Definition 1.4.1]{Ceccherini-Silberstein--Coornaert}. Although it seems to be natural, such a definition doesn't capture all possible maps coming from local rules. More precisely, one could ask about maps $\Phi:\Lambda\to\Gamma$ for which $\big(\Phi(x)\big)_n$ depends only on $x_n\ldots x_{n+\ell}$, but the anticipation $\ell$ is allowed to vary (note that in the finite alphabet case the existence of a local rule implies that it is possible to find a superior bound to the anticipation).
In \cite[Definition 7.1]{Ott_et_Al2014}, the authors present some advances in this direction, by defining sliding block codes between Ott-Tomforde-Willis shift spaces as maps $\Phi:\Lambda\subset\Sigma_A\to \Gamma\subset\Sigma_B$ which are
\begin{description}
\item[a)] continuous at the empty sequence;
\item[b)] for each $a\in A$, there exists $n(a)\in\N$ and a map $\phi^a:B_{n(a)}(\Lambda)\cap Z(a)\to L_\Gamma$, such that for all $x\in\Lambda^{\text{inf}}$ and $i\in\N$ it follows that $\big(\Phi(x)\big)_i=\phi^{x_i}(x_i,\ldots,x_{i+n(x_i)-1})$.
\end{description}
The above definition reflects the fact that a sliding block code could have a local rule with variable anticipation. However, of the two conditions above one is too restrictive while the other is too general. The assumption  that $\Phi$ is continuous at $\O$ is restrictive since it implies that the shift map $\s:\Lambda\subset\Sigma_A\to \Lambda\subset\Sigma_A$ is not a sliding block code, unless $\Lambda$ is a column-finite shift. Furthermore, this continuity also implies that $\Phi(\O)=\O$ (see \cite[Remark 7.2]{Ott_et_Al2014}) which is restrictive from a dynamical point of view   since it makes $\Lambda^{\text{fin}}$ and $\Lambda^{\text{inf}}$ invariant under $\Phi$ (although this feature can be interesting for coding it, there is no reason to consider just such class of dynamical systems). Also, the fact that the anticipation of the local rule used to determine $\big(\Phi(x)\big)_n$ depends exclusively on the symbol $x_n$ is a strong restriction. One can easily construct maps where the anticipation of the local rule used to decide $\big(\Phi(x)\big)_n$ doesn't depend on $x_n$ (for example, take $A=\N$ and $\Phi:\Sigma_A\to\Sigma_A$ given by $\big(\Phi(x)\big)_n:=\min\{x_i: n\leq i\leq n+x_{n+1}\}$ for all $x\in\Sigma_A^{\text{inf}}$). Moreover, the above definition implies that $\Phi(\Lambda^{\text{inf}})\subset\Lambda^{\text{inf}}$, which is also a restriction on the dynamics. On the other hand, the second statement is too general since it imposes a local rule only for infinite sequences, and leaves open the definition of $\Phi$ on $\Lambda^{\text{fin}}\setminus\{\O\}$.

Here we propose a new definition for sliding block codes between shift spaces over countable alphabets. This definition is based on the fact that the standard definition of sliding block code between shifts over finite alphabet is equivalent to saying that $\big(\Phi(\cdot)\big)_n$ is a simple function of the form $$\big(\Phi(x))_n=\sum_{a\in L_\Gamma}a\mathbf{1}_{C_a}\circ\sigma^{n-1}(x),$$ where $\{C_a:\ a\in L_\Gamma\}$ is a finite partition of $\Lambda$, each $C_a$ is a finite union of cylinders defined on the coordinates $1,\ldots, \ell+1$ and  $\sum$ stands for the symbolic sum.

Roughly, a map $\Phi:\Lambda\subset\Sigma_A\to \Gamma\subset\Sigma_B$ will be called a sliding block code if for all $x=(x_i)_{i\in \N}\in\Lambda$ and  $n\in\N$, we have that $\bigl(\Phi(x)\bigr)_n$ depends only on a finite number of entries of $x$, namely
$(x_{n}, \ldots, x_{n+\ell})$, for some $\ell\geq 0$, which doesn't depend on the value of $n$, but depends on the configuration of $x$ from $x_n$ on. In order to give an exact definition for a sliding block code we need the notion of a finitely defined set.

\begin{defn}\label{finitely_defined_set}
Let $\Lambda\subseteq\Sigma_A$ be a set. We say that a set $C\subset \Lambda$ is {\em finitely defined} in $\Lambda$  if there exist $I,J\subseteq \N$, $(\ell_i)_{i\in I},(n_j)_{j\in J}$, where $l_i$ and $n_j$ $\in \N\cup\{0\}$, and $(b_i)_{i\in I},(d_j)_{j\in J}$ with $b_i, d_j\in
B(\Sigma_A)$, such that
$$C=\{x\in \Lambda:\ (x_{1}\ldots x_{1+\ell_i})=b_i\ for\ some\ i\in I\}$$
and
$$C^c=\{x\in \Lambda:\ (x_{1}\ldots x_{1+n_j})=d_j\ for\ some\ j\in J\}.$$

If $C$ is finitely defined in $\Lambda$ and $\ell:=\sup_{i\in I}\ell_i$, we will say that $C$ has {\em anticipation} $\ell$ in $\Lambda$ and in such a case we will also say that $C$ has {\em range}
$\ell+1$ in $\Lambda$. If it is not possible to choose  $(\ell_i)_{i\in I}$ with $\ell<\infty$ in the definition of $C$, then we will say that $C$ is finitely defined set of $\Lambda$ with unbounded anticipation and with unbounded range.
\end{defn}

\begin{rmk} Note that the definition above says that $C$ and $C^c$ have the same property and therefore $C$ is a finitely defined set in $\Lambda$ if, and only if, $C^c$ is a finitely defined set in $\Lambda$. This captures the fact that given
any $x\in\Lambda$, there exist $\ell\geq 0$ such that the knowledge of $(x_1, \ldots, x_{1+\ell})$ provides knowledge of whether or not $x$ belongs to $C$.
 In particular, $\Lambda$ and $\emptyset$ are always finitely defined sets in $\Lambda$.
\end{rmk}

We prove next that the sets which generate the topology of the shift space are finitely defined in the shift space and then give more examples of finitely defined sets.

\begin{lem} Let $\Lambda\subseteq\Sigma_A$ be a shift space.
Let $x\in \Lambda^{\text{fin}}$ and let $F\subset \LA$ be a nonempty finite set. Then $Z(x)\cap\Lambda$ and $Z(x, F)\cap\Lambda$ are finitely defined sets in $\Lambda$. The anticipation of $Z(x)\cap\Lambda$ is $l(x)-1$ and the anticipation of $Z(x, F)\cap\Lambda$ is $l(x)$.
\end{lem}

\begin{proof}
Let $x\in\Lambda^{\text{fin}}$ and $F\subset\LA$ be a finite set. Suppose $l(x)=k>0$, set $b_1=(x_1\ldots x_k)$ and let
$\{d_1, d_2, \dots\}$ be an enumeration of $B_k(\Lambda)\setminus\{(x_1\ldots x_k)\}$ (Notice that $B_k(\Lambda)\setminus\{(x_1\ldots x_k)\}$ could contain $x_1\o \ldots \o$ for example). Then, taking $$
I = \{1\}, \hspace{0.5cm} \ell_1 = k-1, \hspace{0.5cm} b_1 = x, \hspace{1cm}J = \N, \hspace{0.5cm} n_j = k -1
$$
Definition \ref{finitely_defined_set} is verified for $Z(x)$.

Now we turn to $Z(x,F)$. Let $\{a_1,a_2,\ldots\}$ be an enumeration of $\tilde{A}\setminus F$. Take $I = J = \N$, and for each $i\in I$ let $b_i:=(x_1\ldots x_k a_i)$, and let $\{d_1,d_2,\ldots\}$ be an enumeration of $B_{k+1}(\Lambda)\setminus\{b_1,b_2,\ldots\}$. Hence $I$, $J$, $\{b_1,b_2,\ldots\}$, $\{d_1,d_2,\ldots\}$
and $\ell_i =n_j = k$, for all $i\in I$ and $j\in J$, verify Definition \ref{finitely_defined_set} for $Z(x,F)$.

Finally, suppose $l(x)=0$. In this case the proof for $Z(x,F)$ follows by taking $\{b_1,b_2,\ldots\}$ an enumeration of $\tilde{A}\setminus F$, $\{d_1,d_2,\ldots d_{|F|}\}$ an enumeration of $F$, $I=\N$, $J=\{1,\ldots,|F|\}$ and $l_i=n_j=0$ for all $i\in I$ and $j\in J$.
\end{proof}

\begin{ex} Below we give more examples and counterexamples of finitely defined sets:

\begin{enumerate}
\item[a)]Finite unions of cylinders in $\Sigma_A$ are finitely defined sets in $\Sigma_A$ with anticipation equal to the maximum anticipation among all cylinders which are taken in that union.

\item[b)] Let $0\leq k<\ell$, $F\subset A$ be a finite set and fix $x\in\Sigma_A^{\text{fin}}$ with $l(x):=\ell$. Define the set $C\subset\Sigma_A$ by $$C:
=\{w\in\Sigma_A: w_i=x_i\text{ for } k+1\leq i\leq\ell,\ w_{\ell+1}\notin F\}.$$ Then $C$ is a finitely defined set in $\Sigma_A$ with anticipation $\ell$.

\item[c)] If $A=\N$ then $C\subset\Sigma_A$, defined by $C=\bigcup_{k\in\N}\{(x_i)_{i\in\N}:\ x_i=k\ if\ i\leq k\}$, is a finitely defined set in $\Sigma_A$ with unbounded anticipation.

\item[d)] $\Sigma_A^{\text{fin}}$ and $\Sigma_A^{\text{inf}}$ are not finitely defined sets in $\Sigma_A$.

 \item[e)] If $\Lambda\subsetneq\Sigma_A$, $\Lambda \neq \{\O\}$, is a shift space then $\Lambda$ is not a finitely defined set in $\Sigma_A$ (to see this use the characterization of shift spaces in terms of forbidden words).

\item[f)] If $A=\N$ then $C\subset\Sigma_A$ defined by $C=\bigcup_{k\in\N}\{(x_i)_{i\in\N}:\ x_k=k\}$ is not a finitely defined set in $\Sigma_A$.

\item[g)] Any subset of $\Sigma_A$ that contains only infinite sequences is not finitely defined in $\Sigma_A$.

\item[h)] If $\Lambda$ is any shift space and $W\subset \Lambda$ any set such that either $\sup_{x\in W}l(x)<\infty$ or $\sup_{x\in W^c}l(x)<\infty$, then $W$ is a finitely defined set in $\Lambda$. In particular, if $x\in \Lambda^{\text{fin}}$, then $\{x\}$ is finitely defined in $\Lambda$.

\end{enumerate}
\end{ex}

\begin{rmk}
A finite union or intersection of finitely defined sets in some shift space $\Lambda$ is also a finitely defined set in $\Lambda$.
\end{rmk}

\begin{rmk}
Note that while for standard shifts over finite alphabets the class of finitely defined sets of a given shift space coincides with the class of clopen sets of that shift space (and so with the class of sets obtained as finite union of cylinders), the same is not true for the infinite alphabet case. In fact, for one-sided shift spaces over infinite alphabets, any clopen set is finitely defined, but most finitely defined sets are not clopen.

For example, if $a\in A$ and $C:=\{(x_i)_{i\in\N}\subset\Sigma_A:\ x_2=a\}$, then $C$ is an open finitely defined set which is not closed, while $C^c$ is a closed finitely defined set which is not open.
\end{rmk}

Finally recall that if $\Lambda\subset\Sigma_A$ is a shift space and $x\in\Lambda$ then, for all $n\in\N$, we have that $x_n=(\s^{n-1}(x))_1$.\\

\begin{defn}\label{defn_sliding block code}
Let $A$ and $B$ be two alphabets and let $\Lambda \subset \Sigma_A$ be a shift space. Suppose $\{C_a\}_{a\in B\cup\{\o\}}$ is a partition of $\Lambda$ such that:
\begin{description}\addtolength{\itemsep}{-0.5\baselineskip}

\item[\em 1.] for each $a\in B\cup\{\o\}$ the set $C_a$ is a finitely defined  set in $\Lambda$;

\item[\em 2.] $C_{\o}$ is shift invariant (that is, $\s(C_{\o})\subset C_{\o}$).

\end{description}

We will say that a map $\Phi:\Lambda\to\Sigma_B$ is a {\em sliding block code} if

\begin{equation}\label{LR_block_code}\bigl(\Phi(x)\bigr)_n=\sum_{a\in B\cup\{\o\}}a\mathbf{1}_{C_a}\circ\sigma^{n-1}(x),\quad \forall x\in\Lambda,\ \forall n\in\N, \end{equation} where $\mathbf{1}_{C_a}$ is the
characteristic function of the set $C_a$ and $\sum$ stands for the symbolic sum.

Let $\Phi$ be a sliding block code and $\ell_a$ the anticipation of each $C_a$. Define $\ell:=\sup_{a\in
B\cup\{\o\}}\ell_a$. If $\ell<\infty$ we will say that $\Phi$ is a {\em $\ell+1$-block code} with anticipation $\ell$, while if $\ell=\infty$ we will say that $\Phi$ has {\em unbounded anticipation}.
\end{defn}

\begin{rmk} Note that in the case that $\Lambda$ is a shift space over a finite alphabet, the above definition coincides with the standard definition of sliding block codes between shift spaces over finite alphabets.

\end{rmk}

\begin{rmk} Notice that, from Definition \ref{defn_sliding block code},
 if $\bigl(\Phi(x)\bigr)_n=\o$ then $\bigl(\Phi(x)\bigr)_m=\o$ for all $m\geq n$. This is required to avoid that $\o$ appears between symbols of $B$ in $\Phi(x)$.
\end{rmk}

\begin{rmk}\label{defn_alpha}
Equation \eqref{LR_block_code} means that any sliding block code $\Phi:\Lambda\to\Sigma_B$ can be defined in terms of a local rule $\alpha:\Omega\cup\Upsilon\to B\cup\{\o\}$. In fact, $\Phi:\Lambda\to\Sigma_B$ is a sliding block code if, and only if, there exist $\Omega,\Upsilon\subset B(\Lambda)$ such that

\begin{itemize}
\item $\Omega\cap\Upsilon=\emptyset$;

\item for each $a\in\Lambda$ there exists $w\in\Omega\cup\Upsilon$ which is its prefix;

\item each $w\in\Omega\cup\Upsilon$ is not a prefix of another word in $\Omega\cup\Upsilon$;

\item if $w\in\Upsilon$ then any suffix of $w$ also belongs to $\Upsilon$;

\end{itemize}
and there exists a function $\alpha:\Omega\cup\Upsilon\to B\cup\{\o\}$ such that $\alpha^{-1}(\o)=\Upsilon$ and, for all $x\in\Lambda$ and $n\in\N$,
\begin{equation}\label{block_map}
\bigl(\Phi(x)\bigr)_n = \alpha(x_nx_{n+1}\ldots x_{n+\ell}),\end{equation}
where $\ell\geq 0$ is the unique integer such that $(x_nx_{n+1}\ldots x_{n+\ell})\in\Omega\cup\Upsilon$ (whose existence and uniqueness is assured by the properties assumed for $\Omega$ and $\Upsilon$).
According to this formalism, we will say that $\Phi$ is a {\em $M$-block code} if $\Omega\cup\Upsilon= B_M(\Lambda)$.
\end{rmk}

\begin{ex}\label{sliding block code_examp} Next we give some examples and counterexamples of sliding block codes:\\

a) Let $A$ be any countable set and $\Lambda\subseteq\Sigma_A$ be a shift space. Then the shift map $\s:\Lambda\to\Lambda$ is a sliding block code with anticipation 1. In fact, by defining $C_a:=\{x\in\Lambda:\ x_2=a\}$ for all $a\in\LA\cup\{\o\}$, it follows that the sliding block code $\Phi:\Lambda\to\Lambda$, given by $\bigl(\Phi(x)\bigr)_n=\sum_{a\in \LA\cup\{\o\}}a\mathbf{1}_{C_a}\circ\sigma^{n-1}(x)$, coincides with the shift map.
We recall that the shift map is an example of a sliding block code which is not invertible. Furthermore, the shift map is continuous if and only if $\Lambda$ is a column-finite shift (see Theorem \ref{general-CHL-Theo_1}).\\

b) Let
$A=\N$ and $\{A_i\}_{i\in\N}$ be a partition of $A$ by non-singleton finite sets. Define $\Lambda\subseteq\Sigma_A$ to be the row-finite shift $\Lambda:=\{\O\}\cup\bigcup_{i\in \N}A_i^\N$. Let $\Phi:\Lambda\to\Lambda$ be the map given, for all $x\in\Lambda$ and $n\in\N$, by
$$\big(\Phi(x)\big)_n:=\max_{j\geq n} x_j,$$
with the convention that $\o> a$ for all $a\in A$.
It is immediate that $\Phi$ is shift commuting and one can check that $\Phi$ is not continuous. Furthermore, since $\Lambda^{\text{fin}}$ contains only the empty sequence $\O$, it follows that $C_{\o}=\Phi^{-1}(\O)=\{\O\}$, which
is shift invariant and finitely defined ($x\in C_{\o}$ if and only if $x_1=\o$). However, $\Phi$ is not a sliding block code, since for some sequences $x\neq \O$ we need to know all the entries $x_j$, with $j\geq n$, to
determine $\big(\Phi(x)\big)_n$.\\

c) Let $A=\N$ and consider the shift space $\Lambda:=\{(x_i)_{i\in\N}\in\Sigma_A:\ x_{i+1}\geq x_i-1\ \forall i\in\N\}$ with the convention that $\o>a$ for all $a\in A$. Then $\Phi:\Lambda\to\Sigma_A$, given by $\big(\Phi(x)\big)_n:=x_{n +x_n}$, is a non-continuous sliding block code (see Theorem \ref{general-CHL-Theo_1}) with  unbounded anticipation.\\

d)  Let $A=\N$ and define a map $\Phi:\Sigma_A\to \Sigma_A$ by
$$\big(\Phi(x)\big)_n:=\left\{\begin{array}{lcl} (x_n+1)/2 &,\ if& x_n\ is\ odd\\
                                           x_n /2 &,\ if& x_n\ is\ even\\
                                           \o      &,\ if& x_n=\o .\end{array}\right.$$
It follows that $\Phi$ is a sliding block code which is onto but not one-to-one. Furthermore, $\Phi$ is continuous (see Theorem \ref{general-CHL-Theo_1}).\\

e) Let
$A=\N$ and let $\{A_i\}_{i\in\N}$ be a partition of $A$ by non-singleton finite sets. Consider the shift space $\Lambda\subset \Sigma_A$ defined by $\Lambda:=\{\O\}\cup\bigcup_{i\in \N}A_i^\N$. For each $i\in\N$ let $\Phi_i:{A_i}^\N\to {A_i}^\N$ be a sliding block code (which is always continuous since it is defined on a full-shift over a finite alphabet) with anticipation $i$. Then the map $\Phi:\Lambda\to \Lambda$ given, for all $x\in\Lambda$, by $$\Phi(x):=\O\mathbf{1}_{\{\O\}}(x)+\sum_{i\in\N}\mathbf{1}_{A_i}(x)\Phi_i(x),$$ is a continuous sliding block code with unbounded anticipation.\\

f) Let $A:=\Z$ and take $\Lambda\subset\Sigma_A$ to be the shift space with infinite sequences $(x_i)_{i\in\N}$ that satisfy, for all $i\geq 2$, the following conditions: $x_{i}=x_{i-1}+1$ if $x_{i-1}\leq -2$, $x_i\geq 0$ if $x_{i-1}=-1$, and $x_i=x_{i-1}$ if $x_{i-1}\geq 0$ (So $\Lambda$ is the closure of the set of infinite sequences described above). Define $\Phi:\Lambda\to\Sigma_A$ by
$$\big(\Phi(x)\big)_n:=\left\{\begin{array}{lcl} -1 &,\ if& x_n<0\\
                                           \o &,\ if& x_n=0\\
                                           x_n &,\ if& x_n>0\\
                                            0     &,\ if& x_n=\o .\end{array}\right.$$
Then $\Phi:\Lambda\to\Phi(\Lambda)$ is an invertible 1-block code which is not continuous (see Theorem \ref{general-CHL-Theo_2}) and $\Phi^{-1}$ is a sliding block code with unbounded anticipation. \\

g) Let $k\in\N$ and consider the alphabet $A:=\{a\in\Z:\ a\leq k\}$. Let $\Lambda\subset\Sigma_A$ be the shift space whose transition rules are the same as the shift space in example {\em f)}. Define $\Phi:\Lambda\to\Sigma_A$ given by
$$\big(\Phi(x)\big)_n:=\left\{\begin{array}{lcl} 0 &,\ if& x_n<0\\
                                           \o &,\ if& x_n=0\\
                                           x_n &,\ if& x_n>0\\
                                            0     &,\ if& x_n=\o .\end{array}\right.$$
In this case $\Phi$ is an injective 1-block code which is continuous (see Theorem \ref{general-CHL-Theo_2}). Furthermore, $\Phi^{-1}$ is not a sliding block code. In fact, note that $\Phi(\O)=(000\ldots)$, where $(000,\ldots)$ is the constant sequence with the symbol 0. Hence, given a sequence starting with 0s, to compute $\big(\Phi^{-1}(000\ldots)\big)_1$ we need to check whether the sequence is constant or not, what is not a local feature. In fact, $\Phi(\Lambda)$ is not even a shift space (since it does not satisfy the infinite extension property).\\

h) Let $d,k\in\N$, $A=\N$ and $\Lambda\subset\Sigma_A$ be the shift space where the sequences $(x_i)_{i\in\N}$ satisfy the transition rules: $x_{i+1}=x_i-1$, whenever $x_i>k$, and $x_{i+1}$ is any symbol, whenever $x_i\leq k$. Define $\Phi:\Lambda\to\Sigma_A$ by
$$\big(\Phi(x)\big)_n:=\left\{\begin{array}{lll} x_n &,\ &if\ x_n,x_{n+1}\leq k\\
                                                 d &,\ & otherwise.\end{array}\right.$$
We have that $\Phi$ is a continuous 2-block code (see Theorem \ref{general-CHL-Theo_2}), which is not invertible. Furthermore, in this case $\Phi(\Lambda)$ is a shift space.\\

i) Set $d,k\in\N$ with $d>k$, $A=\N$ and let $\Lambda\subset\Sigma_A$ be the shift space where the sequences $(x_i)_{i\in\N}$ satisfy the transition rules: $x_{i+1}=x_i-1$, whenever $x_i>k$, and $x_{i+1}< k$, whenever $x_i\leq k$. Let $\Phi:\Lambda\to\Sigma_A$ be given by
$$\big(\Phi(x)\big)_n:=\left\{\begin{array}{lll} x_n &,\ &if\ x_n\leq k\\
                                                 d &,\ & otherwise.\end{array}\right.$$
We have that $\Phi$ is a continuous 1-block code (see Theorem \ref{general-CHL-Theo_2}) and that $\Phi(\Lambda)$ is a shift space. Furthermore, $\Phi$ is invertible on its image and $\Phi^{-1}$ is a continuous map. Note that $\Phi^{-1}$ is given by
 $$\bigl(\Phi^{-1}(y)\bigr)_n=\sum_{\alpha\in \LA\cup\{\o\}}\alpha\mathbf{1}_{D_\alpha}\circ\s^{n-1}(y)$$
 where
 $$D_\alpha=\left\{\begin{array}{lcl} Z(\alpha)\cap\Phi(\Lambda)                         &,\ if& \alpha\leq k \\\\
                                      Z(\underbrace{dd\ldots d}_{\alpha-k\ times},\{d\})\cap\Phi(\Lambda)  &,\ if& \alpha>k,\ \alpha\neq\O\\\\
                                      \{(ddd\ldots)\}                                    &,\ if& \alpha=\O
                               \end{array}\right. .$$
  Thus, since $D_{\o}$ is not a finitely defined set, it follows that $\Phi^{-1}$ is not a sliding block code.\\

\end{ex}

\subsection{Sliding block codes and continuous shift-commuting maps}\label{CHL-Theo}

The following results are immediate:

\begin{prop}\label{sliding block code->shift_commuting} Any sliding block code commutes with the shift map.
\end{prop}

\begin{proof}
Let $\Phi:\Lambda\to\Gamma$ be a sliding block code, where $\Lambda\subseteq\Sigma_A$ and $\Gamma\subseteq\Sigma_B$.
For all $x\in\Lambda$ and $n\in\N$, it follows that

$$\bigl(\s(\Phi(x))\bigr)_n=\bigl(\Phi(x)\bigr)_{n+1}=\sum_{a\in L_\Gamma\cup\{\o\}}a\mathbf{1}_{C_a}\circ\sigma^n(x)=\sum_{a\in L_\Gamma\cup\{\o\}}a\mathbf{1}_{C_a}\circ\sigma^{n-1}(\s(x))=\bigl(\Phi(\s(x))\bigr)_{n}.$$

\end{proof}

\begin{cor}\label{sliding block code->preserves_period}
If $\Phi:\Lambda\to\Gamma$ is a sliding block code and $x\in\Lambda$ is a sequence with period $p\geq 1$ (that is, such that $\s^p(x)=x$) then $\Phi(x)$ has also period $p$.
\end{cor}

\begin{proof}

Since $\Phi$ commutes with $\s$ and $\s^p(x)=x$, it follows that $\s^p\bigl(\Phi(x)\bigr)=\Phi\bigl(\s^p(x)\bigr)=\Phi(x)$.

\end{proof}

\begin{cor}\label{sliding block code->emptysequence goes to constant}
If $\Phi:\Lambda\to\Gamma$ is a sliding block code then $\Phi(\O)$ is a constant sequence (that is, either $\Phi(\O)=\O$ or $\Phi(\O)=(ddd\ldots)$ for some $d\in L_\Gamma$).
\end{cor}

\begin{proof}

Since $\s(\O)=\O$, it follows from Corollary \ref{sliding block code->preserves_period} that $\s\big(\Phi(\O)\big)=\Phi(\O)$, which means that $\Phi(\O)$ is a constant sequence.

\end{proof}

\begin{cor}\label{sliding block code->finite seq goes to finite seq}
If $\Phi:\Lambda\to\Gamma$ is a sliding block code such that $\Phi(\O)=\O$ then the image of a finite sequence, say $x\in\Lambda^{\text{fin}}$, by $\Phi$ is a finite sequence in $\Gamma$ with length no greater than $l(x)$.
\end{cor}

\begin{proof}

If $x\in\Lambda^{\text{fin}}$ then $\s^{l(x)}(x)=\O$ and therefore $\s^{l(x)}\circ\Phi(x)=\Phi\circ\s^{l(x)}=\Phi(\O)=\O$.

\end{proof}

\subsubsection{Analogues of the Curtis-Hedlund-Lyndon Theorem}

In \cite{Ceccherini-Silberstein--Coornaert}, Theorem 1.9.1, the authors prove that for two shift spaces $\Lambda\subset\Sigma_A$ and $\Gamma\subset\Sigma_B$, $\Phi:\Lambda^{\text{inf}}\to\Gamma^{\text{inf}}$ is such that there exists $\ell\geq 0$ and a local rule $\phi:A^{\ell+1}\to B$ such that $\big(\Phi(x)\big)_n=\phi(x_n\ldots x_n+\ell)$ for all $n$ if, and only if, $\Phi$ is uniformly continuous and shift commuting. In other words, \cite{Ceccherini-Silberstein--Coornaert} gives a version of the Curtis-Hedlund-Lyndon Theorem for the case when sliding block codes have local rules with bounded anticipation and the shift spaces were not compacted.

Theorem \ref{general-CHL-Theo_1} and \ref{general-CHL-Theo_2} below give the sufficient and necessary conditions on sliding block codes between one-sided Ott-Tomforde-Willis shifts under which it is possible to obtain a Curtis-Hedlund-Lyndon Theorem.\\

\begin{theo}\label{general-CHL-Theo_1} Let $\Lambda\subset\Sigma_A$ and $\Gamma\subset\Sigma_B$ be two shift spaces.
Suppose that $\Phi:\Lambda\to\Gamma$ is a map such that $\Phi(\O)=\O$ and $C_{\o}:=\Phi^{-1}(\O)$ is a finitely defined set. Then $\Phi$ is continuous and shift commuting if, and only if, $\Phi$ is a sliding block code given by $\bigl(\Phi(x)\bigr)_n=\sum_{a\in L_\Gamma\cup\{\o\}}a\mathbf{1}_{C_a}\circ\s^{n-1}(x)$ such that,
for all $a\in L_\Gamma$, the set $C_a$ is a finite (possibly empty) union of generalized cylinders of $\Lambda$.
\end{theo}

\begin{proof}

\begin{description}
\item[$(\Longrightarrow)$] Suppose that $\Phi$ is continuous and shift commuting. Since $\Phi\circ\s=\s\circ\Phi$ it follows that $\s(C_{\o})\subseteq C_{\o}$.

Notice that, for all $a\in L_\Gamma$, the cylinder $Z(a)\cap \Gamma$ is clopen in $\Gamma$ and so $C_a:=\Phi^{-1}(Z(a)\cap \Gamma)$ is clopen in $\Lambda$. Since $\Lambda$ is compact, it follows that each $C_a$ is a compact set. Hence, whenever $C_a$ is not empty, it can be written as a finite union of generalized cylinders of $\Lambda$. Thus, any $C_a$ is a finitely defined set.
Since by hypothesis $C_{\o}$ is a finitely defined set (empty or not), it follows that $\{C_a\}_{a\in L_\Gamma\cup\{\o\}}$ is a partition of $\Lambda$ into finitely defined sets.

Now, for all $x\in\Lambda$, to determine $(\Phi(x))_1$ it is only necessary to know what set $C_a$ contains $x$, that is, $\bigl(\Phi(x)\bigr)_1=\sum_{a\in L_\Gamma\cup\{\O\}}a\mathbf{1}_{C_a}(x)$. Hence, since
$\Phi\circ\s^n=\s^n\circ\Phi$, we have that
$$\bigl(\Phi(x)\bigr)_n=\bigl(\sigma^{n-1}(\Phi(x))\bigr)_1=\bigl(\Phi(\sigma^{n-1}(x))\bigr)_1=\sum_{a\in L_\Gamma\cup\{\O\}}a\mathbf{1}_{C_a}(\sigma^{n-1}(x)).$$\\

\item[$(\Longleftarrow)$] For the converse, suppose that $\Phi:\Lambda\to\Gamma$ is a sliding block code where, for all $a\in L_\Gamma$, the set $C_a$ is either empty or a finite union of generalized cylinders of $\Lambda$. Since Proposition \ref{sliding block code->shift_commuting} assures that $\Phi$ commutes with the shift map, we just need to check that $\Phi$ is continuous. The assumption $\Phi(\O)=\O$ means that $\O\in C_{\o}$ and, therefore, for each $a\in L_\Gamma$, the set $C_a$ is a union of cylinders of the form $Z(x,F)$ with $x\neq \O$.

    Given $\bar x\in\Lambda$, let $(x^i)_{i\geq 1}$ be a sequence converging to $\bar x$ -- we shall check that $\Phi(x^i)\to\Phi(\bar x)$. Without loss of generality, we will assume that: if $x^i\to\O$, then $x_1^i\neq x_1^j$ for all $i\neq j$; if $x^i\to \bar x\in\Lambda^{\text{fin}}\setminus\{\O\}$, then $x^i_n=\bar x_n$ for all $n\leq l(\bar x)$, and $x^i_{l(\bar x)+1}\neq x^j_{l(\bar x)+1}$ for all $i\neq j$; if $x^i\to \bar x\in\Lambda^{\text{inf}}$, then $x^i_n=\bar x_n$ for all $n\leq i$.  Hence, defining $N_{\bar x}:=l(\Phi(\bar x))$ (we notice that $0\leq N_{\bar x}\leq l(\bar x)$ since $\Phi(\O)=\O$), we need to consider the three cases below:

    \begin{description}

            \item[i. $N_{\bar x}=0$:] This case is equivalent to $\bar x\in C_{\o}$, and we shall consider three subcases:

             Suppose $\bar x=\O$. If $\Phi(x^i)=\O$, except for a finite number of indexes $i$, then we directly have that $\Phi(x^i)\to\O$. Suppose that there is a subsequence $(x^{i_k})_{k\geq 1}$ such that $\Phi(x^{i_k})\neq \O$ for all $k\geq 1$. Then, since each $x^{i_k}$ starts with a symbol different from the others and each $C_a$ is a finite union of cylinders of the form $Z(x,F)$, with $x\neq \O$, we have that it is not possible that an infinite number of points $x^{i_k}$ are contained in a given set $C_a$. In other words, the sequence $(a_k)_{k\geq 1}$ of symbols of $L_\Gamma$, defined by $a_k:=\big(\Phi(x^{i_k})\big)_1$ for all $k\geq 1$, has no infinite repetition of any symbol. Therefore, $\Phi(x^{i_k})\to\O$ as $k\to\infty$.

             Secondly, suppose $\bar x\in\Lambda^{\text{fin}}\setminus\{\O\}$. If there are only a finite number of indices $i$ such that $x^i\notin\ C_{\o}$, then it is direct that $\Phi(x^i)\to\O$ as $i\to\infty$. If that is not the case, it is not possible that infinitely many elements $x^{i}$ belong to the same $C_a$. In fact, since $x^{i}_n=\bar x_n$ for all $n\leq l(\bar x)$ and $i\geq 1$, and $C_a$ is a finite union of cylinders of the form $Z(x,F)\cap\Lambda$ with $x\neq\O$, it follows that $x^{i}\in C_a$ for infinitely many indexes $i$ if, and only if, $Z(\bar x,F)\cap\Lambda\subset C_a$. But this means that $\bar x\in C_a$ and then $\Phi(\bar x)\neq \O$. Hence, denoting by $C_{a_i}$ the set containing $x^i$, we have that in the sequence $(a_i)_{i\geq 1}$ there are no infinite repetitions and, since $\big(\Phi(x^i)\big)_i=a_i$, this implies that $\Phi(x^i)\to\O$ as $i\to\infty$.

             Finally, suppose $\bar x\in\Lambda^{\text{inf}}$. As in the previous case, when $x^i\notin\ C_{\o}$ only for a finite number of indices $i$ it is direct that $\Phi(x^i)$ converges to $\O$. Also as before, each set $C_a$ can contain only a finite number of elements $x^i$. Indeed, suppose by contradiction that there exists a subsequence $(x^{i_k})_{k\geq 1}$ and $a\in L_\Gamma$ such that $x^{i_k}\in C_a$ for all $k\geq 1$. Since $x^{i_k}_n=\bar x_n$ for all $n\leq i_k$ and $k\geq 1$, and since $C_a$ is a finite union of cylinders of the type $Z(x,F)\cap\Lambda$ with $x\neq \O$, it follows that $C_a$ should contain a cylinder $Z(\bar x_1\ldots\bar x_t,F)\cap\Lambda$, for some $t\geq 1$ and $F\not\ni \bar x_{t+1}$. But this means that $\bar x\in C_a$ what contradicts that $\Phi(\bar x)=\O$. Hence we conclude that $\Phi(x^i)\to\O$ using the same argument used in the case $\bar x\in\Lambda^{\text{fin}}\setminus\{\O\}$.

            \item[ii. $0<N_{\bar x}<\infty$:] We remark that since $x^i\to \bar x$ then for all $1\leq n\leq l(\bar x)+1$ we have $\s^{n-1}(x^i)\to \s^{n-1}(\bar x)$.

            Let $\bar y=(\bar y_n)_{n\leq N_{\bar x}}:=\Phi(\bar x)$. We shall prove that for each finite subset $F\subset L_\Gamma$ there exits $I_F$ such that $\Phi(x^i)\in Z(\bar y,F)\cap\Gamma$ for all $i\geq I_F$.

            For each $1\leq n\leq N_{\bar x}$ we have that $\s^{n-1}(\bar x)\in C_{\bar y_n}$, so
            %(since $C_{\bar y_n}$ is a union of cylinders and $\s^{n-1}(x^i)$ converges to $\s^{n-1}(\bar x)$)
             there exists $I(n)\geq 1$ such that $\s^{n-1}(x^i)\in C_{\bar y_n}$ for all $i\geq I(n)$ and then $\big(\Phi(x^i)\big)_n=\bar y_n$ for all $i\geq I(n)$. Therefore, taking $I_1:=\max_{1\leq n\leq N_{\bar x}} I(n)$ it follows that $\big(\Phi(x^i)\big)_n=\bar y_n$ for all $i\geq I_1$ and $1\leq n\leq N_{\bar x}$.

            On the other hand, $\s^{N_{\bar x}}(\bar x)\in C_{\o}$ and, since $\s^{N_{\bar x}}(x^i)\to \s^{N_{\bar x}}(\bar x)$ from subcase {\em i.} above, it follows that $\Phi(\s^{N_{\bar x}}(x^i))\to \Phi(\s^{N_{\bar x}}(\bar x))=\O$. This means that for any finite set $F\subset L_\Gamma$, there exits $I_F\geq I_1$ such that $\Phi(\s^{N_{\bar x}}(x^i))\in Z(\O,F)\cap\Gamma$ for all $i\geq I_F$. Hence, we have that $\Phi(x^i)\in Z(\bar y,F)\cap\Gamma$ for all $i\geq I_F$.

            \item[iii. $N_{\bar x}=\infty$:] In this case necessarily $l(\bar x)=\infty$. Hence we can use the same argument in the first part of subcase {\em ii.} above.

            Let $\bar y=(\bar y_n)_{n\geq 1}:=\Phi(\bar x)$. Since $x^i\to \bar x$ then, for all $n\geq 1$, we have that $\s^{n-1}(x^i)\to \s^{n-1}(\bar x)$. So , for each $n\geq 1$, there exists $I(n)$ such that $\s^{n-1}(x^i)\in C_{\bar y_n}$ for all $i\geq I(n)$. Hence $\big(\Phi(x^i)\big)_n=\bar y_n$ for all $i\geq I(n)$. Given $Z(\bar y_1\ldots \bar y_K)\cap\Gamma$ a neighborhood of $\bar y$, we can take $I:=\max_{1\leq n\leq K}I(n)$ and for all $i\geq I$ we have that $\Phi(x^i)\in Z(\bar y_1\ldots \bar y_K)\cap\Gamma$. Thus $\Phi(x^i)\to\bar y$.
        \end{description}
\end{description}
\end{proof}

The above theorem characterizes shift commuting continuous maps which map the empty sequence to the empty sequence. Notice that this excludes maps from infinite-alphabet shift spaces to finite-alphabet ones, since finite-alphabet shift spaces do not contain the empty sequence. What our next result shows is that, essentially, this is the only case which is excluded.

\begin{theo}\label{general-CHL-Theo_2} Let $\Lambda\subset\Sigma_A$ and $\Gamma\subset\Sigma_B$ be shift spaces. Take $d\in B$ and suppose $\Phi:\Lambda\to\Gamma$ is a map such that $\Phi(\O)=(ddd\ldots)$ and $C_{\o}:=\Phi^{-1}(\O)$ is finitely defined in the case that $\O\in\Gamma$. Then $\Phi$ is continuous and commutes with the shift map if, and only if, $\Phi$ is a sliding block code given by $\bigl(\Phi(x)\bigr)_n=\sum_{a\in L_\Gamma\cup\{\o\}}a\mathbf{1}_{C_a}\circ\s^{n-1}(x)$ such that,
for all $a\in L_\Gamma$, the set $C_a$ is a finite (maybe empty) union of generalized cylinders of $\Lambda$ with the following properties:
\begin{enumerate}
\item[1 -] $C_a$ is empty for all except a finite number of $a\in L_\Gamma$; and,
\item[2 -] for each $M\geq 1$ there exists a finite set $F_M\subset \LA$ such that $\s^{n-1}\big(Z(\O,F_M)\cap\Lambda\big)\subset C_d$, for all $1\leq n\leq M$.
\end{enumerate}

\end{theo}

\begin{proof}{\color{white}.}

\begin{description}

\item[$(\Longrightarrow)$] Suppose that $\Phi: \Lambda\to \Gamma$ is continuous and commutes with the shift. To prove that equation \eqref{LR_block_code} holds, with each $C_a$ being a finite union of cylinders in $\Lambda$, and with $\s(C_{\o})\subset C_{\o}$, we use the same ideas as in the proof of Theorem \ref{general-CHL-Theo_1}. Since $C_d$ is a finite union of cylinders in $\Lambda$ and $\O\in C_d$, then $C_d=\big(Z(\O,F_0) \cup \bigcup_{i=1}^{k}Z(z^i,F_i)\big)\cap\Lambda$, where $z^i\in\Sigma_A^{\text{fin}}$ and $F_i\subset \LA$ for each $i=0,\ldots,k$.

    We show that there can be only a finite number of sets $C_a$ which are not empty. To this end suppose, by contradiction, that there exists $\{a_i: i\in\N\}\subset L_\Gamma$ such that $a_i\neq a_j$, for $i\neq j$, and $C_{a_i}\neq\emptyset$. For each $i\in\N$, take $x^i\in C_{a_i}$. Since $\big(\Phi(x^i)\big)_1=a_i$ it follows that $\lim_{i\to\infty}\Phi(x^i)=\O$. Now, let $(x^{i_k})_{k\in\N}$ be a convergent subsequence of $(x^i)_{i\in\N}$ and let $\bar x$ denote its limit. From the continuity of $\Phi$ we get that $\bar x\in C_{\o}$. Observe that this means that $\bar x\notin\Lambda^\text{fin}$, since $l(\bar x)<\infty$ implies that $\big(\Phi(\bar x)\big)_{l(\bar x)+1}= d$, which contradicts that $\bar x \in C_{\o}$. Furthermore, since $C_{\o}$ is a finitely defined set it follows that there exists $\ell\geq 0$ such that if $y\in\Lambda$ is such that $y_1\ldots y_{1+\ell}=\bar x_1\ldots \bar x_{1+\ell}$ then $y\in C_{\o}$. But $x^{i_k}$ converges to $\bar x$ if and only if there exists $K\in\N$ such that $x^{i_k}_1\ldots x^{i_k}_{1+\ell}=\bar x_1\ldots \bar x_{1+\ell}$ for all $k\geq K$. Hence we have that $x^{i_k}\in C_{\o}$ for all $k\geq K$, which contradicts the fact that $x^{i_k}\in C_{a_{i_k}}$ for all $k$.

    Now we shall prove that, for each $M\geq 1$, there exists a finite set $F_M\subset \LA$ such that $\s^n\big(Z(\O,F_M)\cap\Lambda\big)\subset C_d$, for all $n\leq M$. Again by way of contradiction, suppose that this property does not hold. Then there should exist $n\in\N$ such that for any finite set $F\subset \LA$ we have $\s^n\big(Z(\O,F)\cap\Lambda\big)\not\subset C_d$. Let $\{a_i:\ i\in\N\}$ be an enumeration of $\LA\setminus F_0$ and define $F^k:=F_0\cup\{a_i:\ i\leq k\}$. It follows that $Z(\O,F^{k+1})\cap\Lambda\subsetneq Z(\O,F^k)\cap\Lambda\subsetneq Z(\O,F_0)\cap\Lambda$. For each $k\in\N$, take $x^k\in Z(\O,F^k)\cap\Lambda$ such that $\s^n(x^k)\notin C_d$. It is clear that $x^k$ converges to $\O$. Due to the continuity of $\Phi$ it follows that $\Phi(x^k)$ converges to the constant sequence $(ddd\ldots)$. However, for all $k\in\N$, we have that $\s^n(x^k)\notin C_d$ and hence, for all $k\in\N$, we have $\Big(\Phi(x^k)\Big)_{n+1}=\Big(\s^{n}\big( \Phi(x^k)\big)\Big)_1 =\Big(\Phi\big(\s^{n} (x^k)\big)\Big)_1
    \neq d$, a contradiction.

\item[$(\Longleftarrow)$] Let $d\in B$ and suppose $\Phi$ is a sliding block code where $\Phi(\O)=(ddd\ldots)$ and, for all $a\in L_\Gamma$, the set $C_a$ is a finite (maybe empty) union of generalized cylinders in $\Lambda$. Furthermore, suppose that the properties {\em 1.} and {\em 2.} hold.

    As before, we have that $C_d=\big(Z(\O,F_0) \cup \bigcup_{i=1}^{k}Z(z^i,F_i)\big)\cap\Lambda$, where $z^i\in\Sigma_A^{\text{fin}}$ and $F_i\subset \LA$ for each $i=0,\ldots,k$. Furthermore, since for $a\in L_\Gamma$ the set $C_a$ is clopen and there are only a finite number of nonempty $C_a$, it follows that $C_{\o}=\Lambda\setminus\bigcup_{a\in L_\Gamma}C_a$ is clopen and then it is also a finite (maybe empty) union of cylinders of $\Lambda$.

    Given $\bar x\in\Lambda$, let $(x^i)_{i\geq 1}$ be a sequence converging to $\bar x$. We shall check that $\Phi(x^i)\to\Phi(\bar x)$. Without loss of generality, we will assume the following: if $x^i\to\O$ then $x_1^i\neq x_1^j$, for all $i\neq j$; if $x^i\to \bar x\in\Lambda^{\text{fin}}\setminus\{\O\}$ then $x^i_n=\bar x_n$, for all $n\leq l(\bar x)$, and $x^i_{l(\bar x)+1}\neq x^j_{l(\bar x)+1}$ for all $i\neq j$; and if $x^i\to \bar x\in\Lambda^{\text{inf}}$ then $x^i_n=\bar x_n$, for all $n\leq i$.  As before, we define $N_{\bar x}:=l(\Phi(\bar x))$. We have two cases:

    \begin{description}

            \item[i. $N_{\bar x}<\infty$ :] This implies that $\bar x\in \Lambda^{\text{inf}}$ because, if $l(\bar x) <\infty$, then for all $n>l(\bar x)$ we would have that $\big(\Phi(\bar x)\big)_n=\big(\s^{n-1}\circ\Phi(\bar x)\big)_1= \big(\Phi\circ\s^{n-1}(\bar x)\big)_1=\big(\Phi(\O)\big)_1=d$, which contradicts that $N_{\bar x}<\infty$.

                If $\Phi(\bar x)=\O$ then $\bar x\in C_{\o}$. Therefore, since $C_{\o}$ is an open set, $x^i\to \bar x$ implies that there exists $K\geq 1$ such that $x^i\in C_{\o}$, for all $i\geq K$ and hence $\Phi(x^i)=\O$ for all $i\geq K$, that is, $\Phi(x^i)\to\O$.

                If $\Phi(\bar x)=(a_1\ldots a_{N_{\bar x}})$ is a finite word not equal to $\O$ then $\s^{n-1}(\bar x)\in C_{a_n}$, for all $1\leq n\leq ,N_{\bar x}$, while $\s^{n-1}(\bar x)\in C_{\o}$ for all $n>N_{\bar x}$. Since each of $C_{a_1},\ldots,C_{a_{N_{\bar x}}},C_{\o}$ is a finite union of cylinders in $\Lambda$, then for each $n$ there exists $\ell_n\geq 0$ such that the set $C_a$ which contains $\s^{n-1}(\bar x)$ only depends on the entries $(\bar x_n\ldots \bar x_{n+\ell_n})$. In particular, any point $z\in \Lambda$ with $(z_n\ldots z_{n+\ell_n})=(\bar x_n\ldots \bar x_{n+\ell_n})$ will also be such that $\s^{n-1}(z)$ belongs to the same $C_a$ as $\s^{n-1}(\bar x)$. Hence we just need to take $K:=N_{\bar x}+\max\{\ell_1,\ldots,\ell_{N_{\bar x}}\}$ and then $x^i_m=\bar x_m$ for all $m\leq K$ and $i\geq K$. It follows that,, for all $i\geq K$, $\s^{n-1}(x^i)\in C_{a_n}$ for all $n\leq N_{\bar x}$ and $\s^{N_{\bar x}}(x^i)\in C_{\o}$. Furthermore, since $C_{\o}$ is invariant under $\s$, we have that $\s^{n-1}(x^i)\in C_{\o}$ for all $n>N_{\bar x}$ and $i\geq K$. Therefore for $i\geq K$ we have that

                $$\big(\Phi(x^i))_n=\big(\s^{n-1}\circ\Phi(x^i)\big)_1=\big(\Phi\circ \s^{n-1}(x^i)\big)_1=\left\{
                \begin{array}{lcl} a_n &,\ if& n\leq N_{\bar x}\\\\
                                    \o &,\ if& n> N_{\bar x},
                \end{array}\right.$$
                which means that $\Phi(x^i)\to\Phi(\bar x)$.

            \item[ii. $N_{\bar x}=\infty$:] Let $\Phi(\bar x)=(a_1a_2a_3\ldots)$.
             For each $m\in\N$, define $y^m=(a_1\ldots a_m)$. Then $\big\{Z(y^m)\cap\Gamma\big\}_{m\in\N}$ is a neighborhood basis for $\Phi(\bar x)$.

             If $\bar x\in \Lambda^{\text{inf}}$ then we can use the same argument used in {\em i.} above. That is, given $M\geq 1$ we can take $K:=N_{\bar x}+\max\{\ell_1,\ldots,\ell_M\}$, where $\ell_1,\ldots,\ell_M$ are such that if $z\in\Lambda$ satisfies $(z_n\ldots z_{n+\ell_n})=(\bar x_n\ldots \bar x_{n+\ell_n})$ then $\s^{n-1}(z)\in C_{a_n}$. Then, as before, for all $i\geq K$ and $n\leq M$ we have $\s^{n-1}(x^i)\in C_{a_n}$ which means that $\big(\Phi(x^i)\big)_n=a_n=\big(\Phi(\bar x)\big)_n$ for all $n\leq M$ and hence $\Phi(x^i)\in Z(y^M)\cap\Gamma$. With this we conclude that $\Phi(x^i)\to\Phi(\bar x)$.

             If $\bar x = \O$ then $a_n=d$ for all $j\in\N$. From hypothesis {\em 2.}, given $M\geq 1$ there exists a finite set $F_M\subset \LA$, such that $\s^{n-1}\big(Z(\O,F_M)\cap\Lambda\big)\subset C_d$ for all $1\leq n\leq M$. Since $x^i\to\bar x$ we can take $K\geq 1$ such that $x^i\in Z(\O,F_M)$ for all $i\geq K$. Therefore, for all $i\geq K$ and $1\leq n\leq M$ it follows that $\s^{n-1}(x^i)\in C_d$, that is, $\big(\Phi(x^i)\big)_n=\big(\s^{n-1}\circ\Phi(x^i)\big)_1=\big(\Phi\circ\s^{n-1}(x^i)\big)_1= d$. Thus, for all $M\geq 1$ there exists $K\geq 1$ such that $\Phi(x^i)\in Z(y^M)\cap\Gamma$ for all $i\geq K$, which means that $\Phi(x^i)\to \Phi(\bar x)$.

             If $\bar x\in \Lambda^{\text{fin}}\setminus\O$ then we have that $\Phi(\bar x)=(a_1a_2\ldots a_{l(\bar x)}ddd\ldots)$. Thus, for each $n\leq l(\bar x)$, we have $\s^{n-1}(\bar x)\in C_{a_n}$ and, for $n>l(\bar x)$, we have $\s^{n-1}(\bar x)\in C_d$. Since for each $m\leq l(\bar x)$ the set $C_{a_m}$ is a finite union of  cylinders in $\Lambda$ and $\s^{m-1}(\bar x)=(\bar x_m\ldots \bar x_{l(\bar x)})\in C_{a_m}$, then necessarily there exists $G$ such that for each $m\leq l(\bar x)$  we have $Z(\bar x_m\ldots \bar x_{l(\bar x)}, G)\cap\Lambda\subset C_{a_m}$. Therefore given $M\geq 1$, we can take the finite set $F:=G\cup F_M$ and we have that $\s^{m-1}\big((\bar x_1\ldots \bar x_{l(\bar x)}, F)\big)=Z(\bar x_m\ldots \bar x_{l(\bar x)}, F)\cap\Lambda\subset C_{a_m}$ for all $m\leq l(\bar x)$, and $\s^{m-1}\big(Z(\bar x_1\ldots \bar x_{l(\bar x)}, F)\cap\Lambda\big)\subset \s^{m-1}\big(Z(\bar x_1\ldots \bar x_{l(\bar x)}, F_M)\cap\Lambda\big)=\s^{m-l(\bar x)-1}\big(Z(\O,F_M)\big)\subset C_d$ for all $l(\bar x)+1\leq m\leq l(\bar x)+ M$.
             Now let $K\geq 1$ be such that $x^i\in Z(\bar x_1\ldots \bar x_{l(\bar x)}, F)\cap\Lambda$ for all $i\geq K$. Then  for all $i\geq K$ we have that
             $$\big(\Phi(x^i)\big)_m=\big(\s^{m-1}\circ \Phi(x^i)\big)_1=\big(\Phi\circ\s^{m-1}(x^i)\big)_1=\left\{\begin{array}{ll}
             a_m &\text{if } m\leq l(\bar x)\\\\
             d   &\text{if } l(\bar x)+1\leq m\leq l(\bar x)+ M,
             \end{array}\right.$$
             which means that $\Phi(x^i)\in Z(y^{l(\bar x)+M})$. Since this works for all $M\geq 1$, we have $\Phi(x^i)\to \Phi(\bar x)$.

        \end{description}

\end{description}

\end{proof}

We now turn our attention to a special class of sliding block codes, the {\em higher block codes}. Higher block codes are a class of sliding block codes with special importance for coding theory. In the classical theory of finite alphabet shifts, higher block codes allow one to encode $M$-step shifts as 1-step shifts. For the infinite alphabet case we can define higher block codes that will be continuous sliding block codes, but they may fail to be invertible on finite sequences with relatively small length. However, we will show that we can use higher block codes to find a 1-step shift which is a factor of a given $M$-step shift.

\begin{rmk} This difference between the theory of higher block codes for finite/infinite alphabets is to be expected, as in \cite{GR} it is proved that there are ($M+1$)-step shift spaces that are not conjugate to any $M$-step shift space (using the length preserving definition of conjugacy of Ott-Tomforde-Willis).
\end{rmk}

\begin{defn}\label{defn_HBC} Given $\Sigma_A$ and $M\in\N$, denote $A^{(M)}:=B_M(\Sigma_A)$. The {\em $M^{th}$ higher block code}
is the map $\Xi^{(M)}:\Sigma_A\to\Sigma_{A^{(M)}}$ given, for all $x\in\Sigma_A$ and $i\in\N$, by
$$\big(\Xi^{(M)}(x)\big)_i:=\begin{cases}\left[x_i\ldots x_{i+M-1}\right] &\text{if } x_j\neq\o\text{ for all } i\leq j\leq i+M-1\\
                                                \o & \text{otherwise}.\end{cases}$$

\end{defn}

It is immediate that a higher block code is a sliding block code with bounded anticipation. We also have the following.

\begin{prop}\label{HBC}
Let $\Lambda\subset \Sigma_A$ be a shift space and let $\Xi^{(M)}:\Sigma_A\to\Sigma_{A^{(M)}}$ be the $M^{th}$ higher block code. Then,
\begin{enumerate}

\item $\Xi^{(M)}$ is continuous.

\item $\Lambda^{(M)}:=\Xi^{(M)}(\Lambda$) is a shift space in $\Sigma_{A^{(M)}}$.

\item Let $\Lambda^*:=\{x\in\Lambda:\ l(x)=0\ or\ l(x)\geq M\}$. Then $\left.\Xi^{(M)}\right|_{\Lambda^*}:\Lambda^*\to\Lambda^{(M)}$ is an invertible map defined by a local rule (that is, is of the form \eqref{LR_block_code}) whose inverse is given, for all $\mathbf x\in\Lambda^{(M)}$ and $i\in\N$, by

\begin{equation}\label{inverse_HBC}\Big({\Xi^{(M)}}^{-1}\big(\mathbf x\big)\Big)_i=\left\{\begin{array}{lll}
x_i &,\ if& i\leq l(\mathbf x),\ and\ \mathbf x_i=[x_i
\ldots x_{i+M-1}]\\\\
x_{l(\mathbf x) +j} &,\ if& i=l(\mathbf x)+j,\ with\ 1\leq j\leq M-1\ and\ \mathbf x_{l(\mathbf x)}=[x_{l(\mathbf x)}
\ldots x_{l(\mathbf x)+M-1}]\\\\
\o &,\ & otherwise.

\end{array}\right.
\end{equation}

\end{enumerate}
\end{prop}

\begin{proof} We prove each statement separately.

\begin{enumerate}

\item This follows directly from Theorem \ref{general-CHL-Theo_1}.

\item Since $\Xi$ is a continuous sliding block code, $\Lambda^{(M)}$ is closed and shift invariant. Thus we only need to show that $\Lambda^{(M)}$ satisfies the infinite extension property.

  Observe that if $y=\big([b_i\ldots b_{i+M-1}]\big)_{1\leq i\leq L}\in{\Lambda^{(M)}}^{\text{fin}}$ then it must be the image of some finite sequence $b=(b_i)_{1\leq i\leq L+M-1}\in\Lambda^{\text{fin}}$. Therefore, since
   $\Lambda$ has the infinite extension property, there exists a sequence, $x^n = (x_i^n)_{i\in\N}$ in $\Lambda^{\text{inf}}$, such that $x^{n}_i=b_i$ for all $n\geq 1$ and $1 \leq i\leq L+M-1$, and $x^n_{L+M}\neq x^m_{L+M}$ if $n\neq m$ (that
   is, $(x^n)_{n\in\N}$ is a sequence of infinite sequences which converges to $b$).
   Define $y^n=(y^n_i)_{i\in\N}:=\Xi^{(M)}(x^{n})$. It is clear that $y^n\in {\Lambda^{(M)}}^{\text{inf}}$. Note that, for all $n\geq 1$ and $1 \leq i\leq L$, we have that $y^n_i=[b_i\ldots b_{i+M-1}]=y_i$ and $y^n_{L+1}=[b_{L+1}\ldots
   x^n_{L+M}]$. Hence, since $x^n_{L+M}\neq x^m_{L+M}$ if $n\neq m$, it follows that $y^n_{L+1}=[b_{L+1}\ldots x^n_{L+M}]\neq[b_{L+1}\ldots x^m_{L+M}]= y^m_{L+1}$ if $m\neq n$, which means that $y^n$ converges to
   $y$. Since this holds for any sequence in ${\Lambda^{(M)}}^{\text{fin}}$, the infinite extension property holds in $\Lambda^{(M)}$.

\item This follows directly from observing that $\Xi^{(M)}$ is injective on the set of sequences with length equal or greater than $M$, while any sequence with length less than $M$ will be mapped to $\O$.

\end{enumerate}

\end{proof}

 Given a shift space $\Lambda\subseteq\Sigma_A$, we say that the shift space $\Lambda^{(M)}$ is its {\em $M^{th}$ higher block presentation}. Note that from Proposition \ref{HBC} we have that $\Lambda^{(M)}$ is a factor of $\Lambda$. On the other hand, although there is a one-to-one correspondence between $\Lambda^{(M)}$ and $\Lambda^*$ through $\Xi^{(M)}$, this does not mean that  $\Lambda^{(M)}$ and $\Lambda^*$ are conjugate because $\Lambda^*$ is not a shift space in general (since in general it is not $\s$-invariant) and, even if it is, ${\Xi^{(M)}}^{-1}$ may not be a sliding block code (since for some $\mathbf x\in\Lambda$ and $i\in\N$, $({\Xi^{(M)}}^{-1}(\mathbf x))_i$ might depend on the value of $i$).

\begin{cor}
Let $\Lambda\subset \Sigma_A$ be a shift space and let $\Xi^{(M)}:\Sigma_A\to\Sigma_{A^{(M)}}$ be its $M^{th}$ higher block code. Consider the restriction $\left.\Xi^{(M)}\right|_{\Lambda^*}:\Lambda^*\to \Lambda^{(M)}$. The following statements are equivalent:
\begin{enumerate}
\item $\sup_{x\in \Lambda^{\text{fin}}} l(x)<M$ ;
\item $\Lambda^{*\text{fin}}\subset\{\O\}$;
\item The inverse of $\left.\Xi^{(M)}\right|_{\Lambda^*}$ is a sliding block code.
\end{enumerate}
\end{cor}

\begin{proof}

It is direct that $i.$ is equivalent to $ii.$.

To prove that $ii.$ implies $iii.$ we just need to realize that $\Lambda^{*\text{fin}}\subset\{\O\}$ implies that ${\Lambda^{(M)}}^{\text{fin}}\subset\{\O\}$ and then, from \eqref{inverse_HBC}, the inverse of $\left.\Xi^{(M)}\right|_{\Lambda^*}$ is a 1-block code.
Conversely, if the inverse of $\left.\Xi^{(M)}\right|_{\Lambda^*}$ is a sliding block code then, from \eqref{inverse_HBC}, the unique possible finite sequence in $\Lambda^{(M)}$ is the empty sequence, what is only possible if $\Lambda^{*\text{fin}}\subset\{\O\}$.

\end{proof}

\begin{cor}
Let $\Lambda\subset \Sigma_A$ be a shift space and let $\Xi^{(M)}:\Sigma_A\to\Sigma_{A^{(M)}}$ be its $M^{th}$ higher block code. The restriction $\left.\Xi^{(M)}\right|_{\Lambda}:\Lambda\to \Lambda^{(M)}$ is a homeomorphism if, and only if, $\Lambda$ is row-finite.
\end{cor}

\begin{proof}
Note that $\Lambda^*=\Lambda$ if, and only if, $\Lambda$ is row-finite. The result now follows since $\Xi^{(M)}$ is invertible on $\Lambda$ if, and only if, $\Lambda^*=\Lambda$.

\end{proof}

%=======================================================================================================================

%=======================================================================================================================

\section*{Acknowledgments}

\noindent D. Gon\c{c}alves was partially supported by Capes grant PVE085/2012 and CNPq.

\noindent M. Sobottka was supported by CNPq-Brazil grants 304813/2012-5 and 480314/2013-6. Part of this work was
carried out while the author was postdoctoral fellow of CAPES-Brazil at Center for Mathematical Modeling, University of Chile.

\noindent C. Starling was supported by CNPq, and work on this paper occurred while the author held a postdoctoral fellowship at UFSC.

%====================================================== BIBLIOGRAFIA =================================================================

\end{document}